\documentclass[11pt]{article}

\usepackage{mathrsfs}

\usepackage{a4wide}
\usepackage{amssymb}
\usepackage{amsmath}
\usepackage{graphicx}
\usepackage{textcomp}
\usepackage[normalem]{ulem}
\usepackage{mathrsfs}
\usepackage{amsmath}
\usepackage{cases}
\usepackage{color}
\usepackage{algpseudocode}

\usepackage{algorithm}

\newtheorem{theorem}{Theorem}
\newtheorem{lemma}[theorem]{Lemma}
\newtheorem{corollary}[theorem]{Corollary}

\newtheorem{example}[theorem]{Example}

\newenvironment{proof}{{\bf Proof:}}{\hfill$\bull$\medskip}
\newcommand{\bull}{\vrule height 1.8ex width 1.0ex depth 0ex}

\usepackage{color}

\title{Explicit barycentric weights for polynomial interpolation in the roots or extrema of classical orthogonal polynomials}

\author{Haiyong Wang\footnote{Corresponding author: haiyong.wang@cs.kuleuven.be or why198309@yahoo.com.cn},
Daan Huybrechs\footnote{E-mail: daan.huybrechs@cs.kuleuven.be}
and Stefan Vandewalle\footnote{E-mail: stefan.vandewalle@cs.kuleuven.be}\\[3\jot]
        K.U.Leuven\\ Department of Computer Science\\
        Celestijnenlaan 200A, B-3001 Leuven, Belgium.}

\date{\today}

\begin{document}



\maketitle

\begin{abstract}
Barycentric interpolation is arguably the method of choice for
numerical polynomial interpolation. The polynomial interpolant is
expressed in terms of function values using the so-called
barycentric weights, which depend on the interpolation points. Few
explicit formulae for these barycentric weights are known. In [H.
Wang and S. Xiang, Math. Comp., 81 (2012), 861--877], the authors
have shown that the barycentric weights of the roots of Legendre
polynomials can be expressed explicitly in terms of the weights of
the corresponding Gaussian quadrature rule. This idea was
subsequently implemented in the Chebfun package [L. N. Trefethen and
others, The Chebfun Development Team, 2011] and in the process
generalized by the Chebfun authors to the roots of Jacobi, Laguerre
and Hermite polynomials. In this paper, we explore the generality of
the link between barycentric weights and Gaussian quadrature and
show that such relationships are related to the existence of lowering operators
for orthogonal polynomials. We supply an exhaustive list of cases,
in which all known formulae are recovered and also some new formulae are derived, including the barycentric
weights for Gauss-Radau and Gauss-Lobatto points. Based on a fast ${\mathcal O}(n)$ algorithm for the computation of Gaussian quadrature, due to Hale and Townsend, this leads to an ${\mathcal O}(n)$ computational scheme for barycentric weights.

\end{abstract}

\vspace{0.05in}

{\bf Keywords.}  barycentric interpolation formula, Gaussian
quadrature, lowering operators, orthogonal polynomials

\vspace{0.05in}

{\bf AMS subject classifications.} 41A05, 65D05, 65D15

\section{Introduction}\label{s:intro}

Polynomial interpolation is a fundamental tool in many areas of
numerical analysis
\cite{cheney1966approximation,dahlquist2008numerical,davis1963interpolation,gautschi1997numericalanalysis,suli2003,trefethen2012atap}.
It is usually introduced using the Lagrange form of the
interpolating polynomial, as follows. Let $\{x_j\}_{j=0}^{n}$ be a
set of distinct nodes. Then the polynomial of degree $n$ that
interpolates the function $f(x)$ at these points may be written as
\begin{equation}\label{eq:lagrange}
p_n(x)=\sum_{j=0}^{n}f(x_j)\ell_j(x),
\end{equation}
where
\[
\ell_j(x)=\prod_{k\neq j}\frac{x-x_k}{x_j-x_k}, \quad
j=0,1,\ldots,n,
\]
are the Lagrange fundamental polynomials. The Lagrange form of the interpolating polynomial~\eqref{eq:lagrange} is not advocated for numerical computations, as typical algorithms require $\mathcal{O}(n^2)$ operations. Moreover, they are numerically unstable and each time a node $x_j$ is modified or added, all Lagrange fundamental polynomials have to be recalculated~\cite{berrut2004barycentric}.

In order to obtain good approximations via interpolation, the
choice of interpolation nodes is particularly important. For example, it is well-known that equispaced
points give rise to the Runge phenomenon
when the number of interpolation points is large. The interpolating
polynomial diverges, even when the function $f$ to be interpolated
is analytic. In order to avoid the occurrence of Runge's phenomenon,
various techniques have been proposed over the past decades and we
refer the reader to \cite{platte2011impossibility} for a comprehensive discussion. In
practice, the interpolation nodes with the density distribution
$1/\sqrt{1-x^2}$ are optimal in various senses for polynomial
approximation on $[-1,1]$ and they typically lead to well-behaved
Lagrange interpolation \cite{trefethen2012atap}. Good candidates
are the roots and extrema of certain orthogonal polynomials.

Interpolation approximations based on the roots of orthogonal polynomials are widely applied in numerical integration, spectral methods, etc. For computational purposes, an alternative and more preferable approach to the Lagrange form is to rewrite the interpolation polynomial as a sum of the corresponding orthogonal polynomials (see, for example, \cite{davis1984numericalintegration,mason2003chebyshev,shen2006spectral}). Let $\{\pi_n(x)\}_{n\in\mathbb{N}}$ be a sequence of polynomials
orthogonal with respect to a given nonnegative and integrable weight function $\omega(x)$ on $(a,b)$, and
\begin{align}\label{eq:orthog relation}
\int_{a}^{b}\omega(x)\pi_m(x)\pi_n(x)dx=h_n\delta_{mn},\quad
m,n\geq0,
\end{align}
where $h_n$ is a positive normalization constant and $\delta_{mn}$ is the Kronecker delta. Suppose that the interpolation nodes $\{x_j\}_{j=0}^{n}$ are the zeros of $\pi_{n+1}(x)$, then the interpolating polynomial $p_n(x)$ can be written as a linear combination of $\{\pi_k(x)\}_{k=0}^{n}$
\begin{equation}\label{eq:orthog_expansion}
p_n(x)=\sum_{k=0}^{n}a_k\pi_k(x),
\end{equation}
where the coefficients $a_k$ are given by
\begin{equation*}
a_k=h_k^{-1}\sum_{j=0}^{n}w_jf(x_j)\pi_k(x_j), \quad  k=0,\cdots,n,
\end{equation*}
and $\{w_j\}_{j=0}^{n}$ are the Gaussian quadrature weights corresponding to the weight $\omega(x)$. Although this form of the interpolating polynomial is numerically stable, the calculation of the coefficients $\{a_k\}_{k=0}^{n}$ takes $\mathcal{O}(n^2)$ operations by direct evaluation. In the special case of Chebyshev points, the cost can be reduced to $\mathcal{O}(n\log n)$ using the FFT. Except for that case, using the form~\eqref{eq:orthog_expansion} leads to $\mathcal{O}(n^2)$ methods which makes it inefficient for large $n$.

The fast computation of the Lagrange interpolation polynomial has
received substantial attention over the past decades (see
\cite{berrut2004barycentric,dutt1996fastpolynomial,salzer1972chebyshev,wang2012legendre,werner1984interpolation} and references
therein). From a computational point of view, it is recommended to
apply the barycentric representation of the interpolating
polynomial~\cite{berrut2004barycentric}. The barycentric formula, which we shall
review in Section \ref{s:bary}, equation~\eqref{eq:bary_form}, has
several attractive features such as stability and high efficiency.
For example, when the interpolation nodes are Chebyshev points of
the first or second kind, the evaluation of the interpolating
polynomial requires only $\mathcal{O}(n)$ operations
\cite{henrici1982essentials,salzer1972chebyshev}. For other sets of interpolation points, no
explicit formulae for the barycentric weights are known. Direct
computation of the barycentric weights again requires
$\mathcal{O}(n^2)$ operations.

In this paper we shall devote our attention to the study of the
barycentric weights for roots and extrema of the classical
orthogonal polynomials. This study is motivated by the earlier
observation by the first author in \cite{wang2012legendre} that a
simple relationship exists between Gauss-Legendre quadrature weights
$w_j$ and barycentric interpolation weights
$\lambda_j^{\mathrm{Leg}}$ :
\begin{equation}\label{eq:weights_legendre}
\lambda_j^{\mathrm{Leg}}=(-1)^j\sqrt{(1-x_j^2)w_j}, \qquad
j=0,\ldots,n,
\end{equation}
for the case of interpolation in the roots $x_j$ of the Legendre polynomial of degree $n+1$. Formula~\eqref{eq:weights_legendre}
was implemented in the Chebfun package \cite{trefethen2011chebfun4} as part of the \texttt{legpts} routine \cite{hale2012chebfunquadrature}.
Several useful generalizations were identified by the Chebfun authors in the process. In particular, formula~\eqref{eq:weights_legendre} remains
valid for the more general family of Jacobi polynomials. Furthermore, similar identities hold for the roots of Laguerre polynomials,
\begin{equation}\label{eq:weights_laguerre_chebfun}
 \lambda_j^{\mathrm{Lag}} = (-1)^j \sqrt{x_jw_j}, \qquad  j = 0,\ldots,n,
\end{equation}
and Hermite polynomials,
\begin{equation}\label{eq:weights_hermite_chebfun}
 \lambda_j^{\mathrm{H}} = (-1)^j \, \sqrt{w_j}, \qquad j=0,\ldots,n,
\end{equation}
see routines \texttt{legpts}, \texttt{chebpts}, \texttt{jacpts},
\texttt{lagpts} and \texttt{hermpts} in Chebfun. These developments are detailed in \cite{hale2012chebfunquadrature} and \cite[p.151-152]{trefethen2012atap}.

The motivation for formula \eqref{eq:weights_legendre} was based on
an explicit form of the Gauss-Legendre quadrature weights, see
\cite[Thm.~3.1]{wang2012legendre}. The above formulae lead to the
following question: how general is the link between barycentric
weights and weights of Gaussian quadrature? For this purpose, we
describe an alternative analysis leading to
\eqref{eq:weights_legendre}--\eqref{eq:weights_hermite_chebfun}
using the notion of \emph{lowering operators} of orthogonal
polynomials. This yields additional formulae for interpolation at
the extremae of the classical orthogonal polynomials (Jacobi,
Laguerre and Hermite), corresponding to the Radau and Lobatto
variants of Gaussian quadrature.


Hence, the computation of the barycentric weights has been
transformed into the computation of the nodes and weights of
Gaussian quadrature rules, for which we can apply the well-known
Golub-Welsch algorithm \cite{golub1969gauss} in $\mathcal{O}(n^2)$
operations. Faster algorithms have been described with optimal
${\mathcal O}(n)$ computational complexity for specific cases. The
first such algorithm is due to Glaser, Liu and Rokhlin, requiring
only $\mathcal{O}(n)$ operations for the computation of Gaussian
quadrature for the classical polynomials \cite{glaser2007roots}.
More recently, two approaches have been described with the same
optimal complexity, but with much reduced constants, based on
exploiting known asymptotic behaviour of orthogonal polynomials. The
approach of Hale and Townsend includes the case of Gauss-Jacobi
quadrature \cite{hale2012fastgauss}. Bogaert, Michiels and Fostier
describe a similar approach for Gauss-Legendre quadrature
\cite{bogaert2012legendre}. As a result of these developments, it is
evident that the polynomial interpolants in the roots or extrema of
Jacobi polynomials can be computed in only $\mathcal{O}(n)$
operations as well by using their barycentric representations.



This paper is organized as follows. In the next section, we start with some known results about the barycentric interpolation formula and present a derivation of the barycentric weights for the zeros of orthogonal polynomials. In Section~\ref{s:radaulobatto}, we explore the explicit forms of the barycentric weights for the zeros of orthogonal polynomials with some additional points. We give several numerical examples in Section \ref{s:examples} and conclude with some remarks in Section \ref{s:remarks}.

\section{Barycentric interpolation formula}\label{s:bary}

\subsection{First and second barycentric interpolation formula}

In this section we review some facts about the barycentric interpolation formula. Let
\begin{equation}\label{eq:ell}
\ell(x)=(x-x_0)(x-x_1)\cdots(x-x_n)
\end{equation}
be the monic polynomial of degree $n+1$ that vanishes at the interpolation nodes $x_j$. Then the Lagrange form of the interpolating polynomial $p_n(x)$ can be rewritten in barycentric form as
\begin{equation}\label{eq:bary_form}
p_n(x)=\frac{\displaystyle
\sum_{j=0}^{n}\frac{\lambda_j}{x-x_j}f(x_j)} {\displaystyle
\sum_{j=0}^{n}\frac{\lambda_j}{x-x_j}},
\end{equation}
where the barycentric weights are defined by \cite[p.~218]{stiefel1963introduction}
\begin{equation}\label{eq:weight}
\lambda_j=\frac{1}{\Pi_{j\neq k} (x_j-x_k)} = \frac{1}{\ell{'}(x_j)}, \quad j=0,1,\ldots,n.
\end{equation}
Expression~\eqref{eq:bary_form} is the so-called \emph{second form of the barycentric formula}. Due to the division, the barycentric weights can be simplified by cancelling the common factors without altering the result. We will call the resulting weights the simplified barycentric weights.

The \emph{first form of the barycentric formula} is given by
\begin{equation}\label{eq:bary_form1}
 p_n(x) = \ell(x) \sum_{j=0}^n \frac{\lambda_j}{x-x_j}f(x_j),
\end{equation}
with the weights still defined by~\eqref{eq:weight}. A disadvantage
in this case is that common factors of $\lambda_j$ may not be
cancelled, which leads to more complicated expressions later on. On
the other hand, it is shown recently in \cite{webb2012stability} that the second
formula is not stable for (complex) values of $x$ away from the
interpolation interval, whereas the first formula is. For this
reason we include results for the full barycentric weights defined
by~\eqref{eq:weight} as well as the simplified ones.

For convenience, we assume throughout this paper that the interpolation nodes $x_j$ are monotonic and hence, the barycentric weights $\lambda_j$ always have alternating signs. For points inside the interpolation interval, the barycentric formula \eqref{eq:bary_form} has been proved to be numerically stable for any set of interpolating points with a small Lebesgue constant \cite{higham2004barycentric}. For a general set of interpolation nodes, the computation of the barycentric weights $\{\lambda_j\}_{j=0}^{n}$ requires $\mathcal{O}(n^2)$ operations \cite{werner1984interpolation}. However, for several important sets of points such as Chebyshev points of the first and second kind, explicit formulae for these barycentric weights $\lambda_j$ are known. For example, for the Chebyshev points of the first kind
\begin{equation*}
x_j=\cos\left(\frac{2j+1}{2n+2}\pi\right),\quad j=0,1,\ldots,n,
\end{equation*}
the simplified barycentric weights are given
by~\cite[p.~249]{henrici1982essentials}
\begin{equation}\label{eq:weights_firstkind}
\lambda_j^{\mathrm{CH1}}=(-1)^j\sin\left(\frac{2j+1}{2n+2}\pi\right).
\end{equation}
For the Chebyshev points of the second kind
\begin{equation}\label{eq:points_secondkind}
x_j=\cos\left(\frac{j}{n}\pi\right),\quad j=0,1,\ldots,n,
\end{equation}
the simplified barycentric weights are given by \cite{salzer1972chebyshev}
\begin{equation}\label{eq:weights_secondkind}
\lambda_j^{\mathrm{CH2}}=(-1)^j\delta_j,\quad
\delta_j=\bigg\{\begin{array}{cc}
                                          1/2,   & \mbox{$\textstyle  j=0$ or $j=n$},\\
                                          1,   & \mbox{otherwise}.
                                        \end{array}
\end{equation}
Thus, each evaluation of the interpolation formulae $p_n(x)$ for Chebyshev points can be implemented in only $\mathcal{O}(n)$ operations.

\subsection{Explicit barycentric weights in terms of Gaussian quadrature}

Equation~\eqref{eq:weight} shows that the barycentric weights are
given in terms of the derivative of the polynomial $\ell(x)$ that
vanishes at the interpolation nodes. In the following we consider
the case where $\ell(x)$ is an orthogonal polynomial with respect to a weight function $w(x)$.

Let $\{x_j\}_{j=0}^{n}$ be the $n+1$ roots of the polynomial $\pi_{n+1}(x)$ and $k_{n}$ be the leading coefficient of $\pi_n(x)$. The corresponding Gaussian quadrature rule is
\begin{equation*}
\int_{a}^{b}\omega(x)f(x)dx\simeq\sum_{j=0}^{n}w_jf(x_j),
\end{equation*}
where the Gaussian quadrature weights are given by
\cite[p.~97]{davis1984numericalintegration}
\begin{equation}\label{eq:wj_gauss}
w_j=\frac{k_{n+1}h_n}{k_n\pi_{n+1}'(x_j)\pi_{n}(x_j)},
\end{equation}
with $h_n$ defined as in \eqref{eq:orthog relation}. When the roots $x_j$ are used for interpolation purposes, the corresponding barycentric weights can be written in the following form
\begin{equation}\label{eq:weight_derivative}
\lambda_j=\frac{k_{n+1}}{\pi_{n+1}'(x_j)}.
\end{equation}
Combining this with~\eqref{eq:wj_gauss} leads to
\begin{equation}\label{eq:pi_n_factor}
\lambda_j=\frac{k_n}{h_n}\pi_n(x_j)w_j.
\end{equation}
This relation between barycentric weights and Gaussian weights does not immediately lead to faster computations, as one still has to evaluate the orthogonal polynomial $\pi_n$ in all the nodes $x_j$. Based, for example, on recurrence relations for the polynomials, this step still requires $\mathcal{O}(n^2)$ operations.

The basic observation underlying the remainder of this paper is that for the classical polynomials $\pi_n(x_j)$ can be written in terms of $\pi_{n+1}'(x_j)$. Equation~\eqref{eq:weight_derivative} can be used again to remove the $\pi_n(x_j)$ factor from \eqref{eq:pi_n_factor}, and as a result the barycentric weights are only related to the nodes and weights of corresponding Gaussian quadrature rule.

We prepare the setting and establish notation with the following lemma.

\begin{lemma}\label{lem:hypergeometric}
Let $\pi_n(x)$ satisfy the equation of hypergeometric type
\begin{equation}\label{eq:hyper diff equation}
\varphi(x)\pi_n{''}(x)+\phi(x)\pi_n{'}(x)+\nu_n\pi_n(x)=0,
\end{equation}
where $\varphi(x)$ and $\phi(x)$ are polynomials of at most second
and first degree respectively. When
$\nu_n=-n\phi{'}(x)-\frac{n(n-1)}{2}\varphi{''}(x)$, the above
equation has a particular solution $\pi_n(x)$ which is a polynomial
of degree $n$, and all orders of the derivatives of $\pi_n(x)$ have
the following Rodrigues formula
\begin{equation}\label{eq:rodrigues formula}
\pi_n^{(m)}(x)=\frac{A_{mn}B_n}{\varphi^{m}(x)\omega(x)}\frac{d^{n-m}}{dx^{n-m}}[\varphi^n(x)\omega(x)],
\end{equation}
where
\begin{equation*}
A_{mn}=\frac{n!}{(n-m)!}\prod_{k=0}^{m-1}\left(\phi'(x)+\frac{1}{2}(n+k-1)\varphi{''}(x)\right),
\quad A_{0n}=1,
\end{equation*}
and $B_n$ is a normalizing constant
\begin{equation}\label{eq:constant B}
B_n=k_n\prod_{k=0}^{n-1}\left(\phi'(x)+\frac{1}{2}(n+k-1)\varphi{''}(x)\right)^{-1},\quad
B_0=k_0.
\end{equation}
The function $\omega(x)$ is chosen such that
\begin{equation}\label{eq:weight equation}
(\varphi(x)\omega(x))'=\phi(x)\omega(x).
\end{equation}
Moreover, $\pi_n(x)$ is orthogonal with respect to the function
$\omega(x)$.
\end{lemma}
\begin{proof}
See \cite[p. 24]{nikiforov1988specialfunctions}.
\end{proof}

Note that the polynomials satisfying an equation of the form \eqref{eq:hyper diff equation} with the right conditions are precisely the classical orthogonal polynomials: Jacobi polynomials (which include Legendre, Chebyshev and Gegenbauer polynomials), Laguerre polynomials and Hermite polynomials. The following theorem gives the general relation between barycentric weights and Gaussian quadrature.

\begin{theorem}\label{th:weights}
Let $\pi_n(x)$ satisfy the above conditions. Then the barycentric
weights $\lambda_j$ corresponding to the roots of $\pi_{n+1}(x)$ are
given by
\begin{equation}\label{eq:weights_classical}
\lambda_j=\sigma (-1)^j\sqrt{ \frac{ k_{n+1}^2(2n+2) \varphi(x_j)
w_j}{\nu_{2n+2}h_{n+1}}},\quad j=0,1,\ldots,n,
\end{equation}
where $\sigma=+1$ for even $n$ and $\sigma=-1$ for odd $n$.
\end{theorem}
\begin{proof}
Set $\phi_n(x)=\phi(x)+n\varphi'(x)$. It then follows from
\eqref{eq:weight equation} that
\begin{equation}
(\varphi(x)^{n+1}\omega(x)){'}=\phi_n(x)\varphi(x)^n\omega(x).
\end{equation}
Applying the Rodrigues formula \eqref{eq:rodrigues formula} and
noting that $\phi_n(x)$ is a polynomial of degree one, we have
\begin{align*}
\pi_{n+1}(x)&=\frac{B_{n+1}}{\omega(x)}\frac{d^{n+1}}{dx^{n+1}}[\varphi^{n+1}(x)\omega(x)]\nonumber
\\
&=\frac{B_{n+1}}{\omega(x)}\frac{d^{n}}{dx^{n}}[\phi_n(x)\varphi(x)^n\omega(x)]\nonumber
\\
&=\frac{B_{n+1}}{\omega(x)}[\phi_n(x)(\varphi(x)^n\omega(x))^{(n)}+n\phi_n{'}(x)(\varphi(x)^n\omega(x))^{(n-1)}]\nonumber
\\
&=\frac{B_{n+1}}{B_n}[\phi_n(x)\pi_n(x)+\frac{n\phi_n{'}(x)}{A_{1n}}\varphi(x)\pi_n{'}(x)].
\end{align*}
Consequently,
\begin{equation}\label{eq:lowering step one}
\varphi(x)\pi_n{'}(x)=\frac{A_{1n}}{n\phi_n{'}(x)}\left[\frac{B_n}{B_{n+1}}\pi_{n+1}(x)-\phi_n(x)\pi_n(x)\right].
\end{equation}
Recall that $\pi_n(x)$ is orthogonal with respect to the function
$\omega(x)$. It has the three-term recurrence relation
\begin{align}\label{eq:three term recur}
x\pi_n(x)=\alpha_n\pi_{n+1}(x)+\beta_n\pi_n(x)+\gamma_n\pi_{n-1}(x),
\end{align}
where $\alpha_n$, $\beta_n$ and $\gamma_n$ are constants. Using the
orthogonality of $\pi_n(x)$, we can easily find that
\begin{align}
\alpha_n=\frac{k_n}{k_{n+1}}, \quad
\gamma_n=\frac{h_nk_{n-1}}{k_nh_{n-1}}.
\end{align}
Combining \eqref{eq:lowering step one} and \eqref{eq:three term
recur} yields
\begin{equation}\label{eq:lowering step two}
\varphi(x)\pi_n{'}(x)=\frac{A_{1n}}{n\phi_n{'}(x)}\left[\frac{B_n}{B_{n+1}\alpha_n}(x-\beta_n)-\phi_n(x)\right]\pi_n(x)-\frac{A_{1n}B_n\gamma_n}{n\phi_n{'}(x)B_{n+1}\alpha_n}\pi_{n-1}(x).
\end{equation}
Furthermore, taking the definition \eqref{eq:constant B} into
account and using the fact that
\begin{equation*}
A_{1n}=n(\phi'(x)+\frac{n-1}{2}\varphi{''})=-\nu_n, \quad
\phi_n{'}(x)=\phi{'}(x)+n\varphi{''}(x)=-\frac{\nu_{2n+1}}{2n+1},
\end{equation*}
we obtain
\begin{equation}\label{eq:lowering operator}
\varphi(x)\pi_n{'}(x)=-\frac{(2n+1)\nu_n}{n\nu_{2n+1}}\left[\frac{\nu_{2n}\nu_{2n+1}}{2\nu_n(2n+1)}(x-\beta_n)+\phi_n(x)\right]\pi_n(x)+\frac{\nu_{2n}h_nk_{n-1}}{2nk_nh_{n-1}}\pi_{n-1}(x).
\end{equation}
Let $x_j$, $j=0,\ldots,n$, be the roots of $\pi_{n+1}(x)$. Replacing
$n$ with $n+1$ in \eqref{eq:lowering operator}, it follows that
\[
 \varphi(x_j) \pi_{n+1}'(x_j) = \frac{\nu_{2n+2}k_{n}h_{n+1}}{(2n+2)k_{n+1}h_{n}}\pi_n(x_j).
\]
This, together with \eqref{eq:weight_derivative}
and~\eqref{eq:pi_n_factor}, implies that
\[
\lambda_j^2 = \frac{ k_{n+1}^2(2n+2) \varphi(x_j)
w_j}{\nu_{2n+2}h_{n+1}}.
\]
Recalling that the barycentric weights $\lambda_j$ have alternating
signs, expression~\eqref{eq:weights_classical} follows.
\end{proof}

We make some further comments regarding Theorem \ref{th:weights} and its proof, in order to put the result itself, as well as its scope and limitations, in a proper context:
\begin{itemize}
\item The crucial identity that relates $\pi_n'$ to $\pi_n$ and $\pi_{n-1}$ is given by expression \eqref{eq:lowering operator}.
This is an example of a so-called \emph{lowering operator}. Indeed, moving $\pi_n$ to the left hand side in \eqref{eq:lowering operator}
defines an operator that acts on $\pi_n$ and that yields a polynomial of lower degree -- hence the name. We have included a typical
derivation of the lowering operator in the proof for the purpose of being self-contained. A classical reference is \cite{nikiforov1988specialfunctions}.

\item  The derivatives of classical polynomials $\pi_n^{(k)}$ are again orthogonal on the same interval with respect
to the new weight function $\varphi(x)^k\omega(x)$ \cite{nikiforov1988specialfunctions}
. Thus if the interpolation points are roots of $\pi_n^{(k)}$, $n >
k$, then \emph{lowering operators} can also be used to write the
barycentric weights in terms of the nodes and weights of the
Gaussian quadrature with respect to the new weight function.

\item More general lowering operators can be found for other kinds of orthogonal polynomials as well, leading to identities similar to \eqref{eq:lowering operator}. Examples include polynomials orthogonal with respect to the weight function $e^{-V(x)}$ on $[-1,1]$, where $V(x)$ is a polynomial \cite[Eq. (1.5)]{chen1997ladder}, and polynomials orthogonal with respect to the weight function $w(x) e^{-V(x)}$, where $w(x)$ is any of the classical weight functions \cite[Eq. (6.5)]{tracy1994fredholm}. Any such identity immediately gives rise to a relationship between barycentric weights and Gaussian quadrature weights.

\item However, it is important to point out that the lowering operator in general depends on $n$. This dependence is benign in the setting of Theorem \ref{th:weights}, in the sense that all $n$-dependent quantities still appearing in the final result \eqref{eq:weights_classical} have explicit expressions (which we supply further on in \S\ref{ss:explicit}). In particular, these quantities can be evaluated in a number of operations that does not depend on $n$. Unfortunately, this is no longer the case for the lowering operators in \cite{chen1997ladder,tracy1994fredholm}.

\item Furthermore, the algorithms for the fast construction of Gaussian quadrature rules are only applicable for the classical orthogonal polynomials \cite{glaser2007roots,hale2012fastgauss,bogaert2012legendre}. Thus, an ${\mathcal O}(n)$ algorithm for the computation of barycentric weights from the Gaussian weights can not be immediately generalized to other polynomials.
\end{itemize}
The last two comments are the two reasons for restricting ourselves to the classical orthogonal polynomials in Lemma \ref{lem:hypergeometric} and Theorem \ref{th:weights} above.

\subsection{Explicit barycentric weights for the classical orthogonal polynomials}\label{ss:explicit}

As mentioned above, polynomials satisfying an equation of the form~\eqref{eq:hyper diff equation} are the classical polynomials. For Jacobi polynomials, we have
\cite[p.~61]{szego1939polynomials}
\begin{equation}\label{eq:id_jacobi1}
(1-x^2)P_{n}^{(\alpha,\beta)}{''}(x)+[\beta-\alpha-(\alpha+\beta+2)x]P_{n}^{(\alpha,\beta)}{'}(x)+n(n+\alpha+\beta+1)P_{n}^{(\alpha,\beta)}(x)=0,
\end{equation}
from which we obtain
\begin{equation}\label{eq:u_and_C}
\varphi(x)=1-x^2, \quad \nu_n=n(n+\alpha+\beta+1).
\end{equation}
Substituting these into \eqref{eq:weights_classical} and recalling
\begin{align*}
k_{n+1}=\frac{1}{2^{n+1}}\binom{2n+\alpha+\beta+2}{n+1}, \quad
h_{n+1}=\frac{2^{\alpha+\beta+1}}{2n+\alpha+\beta+3}\frac{\Gamma(n+\alpha+2)\Gamma(n+\beta+2)}{(n+1)! \, \Gamma(n+\alpha+\beta+2)},
\end{align*}
we obtain the barycentric weights. Cancelling the common factors
yields the simplified weights. In the next three corollaries,
$\sigma$ is defined as in Theorem~\ref{th:weights}.

\begin{corollary}\label{cor:jacobi}
 The barycentric weights for the roots of the Jacobi polynomial $P_{n+1}^{(\alpha,\beta)}(x)$ are
\begin{equation}\label{eq:weights_jacobi}
\lambda_j^{(\alpha,\beta)} = C^{(\alpha,\beta)}_n \, (-1)^j \,
\sqrt{(1-x_j^2)w_j}, \qquad j=0,\ldots,n,
\end{equation}
where $C^{(\alpha,\beta)}_n = 1$ for the simplified weights, and
otherwise
\begin{equation}\label{eq:jacobi common factor}
C^{(\alpha,\beta)}_n = \sigma
\frac{\Gamma(2n+\alpha+\beta+3)}{2^{n+1+\frac{\alpha+\beta+1}{2}}}
\frac{1}{\sqrt{(n+1)! \,
\Gamma(n+\alpha+\beta+2)\Gamma(n+\alpha+2)\Gamma(n+\beta+2)}}.
\end{equation}
\end{corollary}

The case of Legendre polynomials corresponds to the choice
$\alpha=\beta=0$. Formula \eqref{eq:weights_jacobi} indeed
corresponds to expression \eqref{eq:weights_legendre}, which was
observed earlier in \cite{wang2012legendre}, for the case of the
simplified weights. Note that the formula for the simplified weights
remains unchanged for the general Jacobi case.

Similarly, the Laguerre polynomials satisfy the following
differential equation \cite[p.~100]{szego1939polynomials}
\begin{equation*}
xL_{n}^{(\alpha)}{''}(x)+(1+\alpha-x)L_{n}^{(\alpha)}{'}(x)+nL_{n}^{(\alpha)}(x)=0,
\end{equation*}
and
\begin{equation*}
k_{n+1}=\frac{(-1)^{n+1}}{(n+1)!}, \quad
h_{n+1}=\frac{\Gamma(n+\alpha+2)}{(n+1)!}.
\end{equation*}
These lead to the following expressions.

\begin{corollary}\label{cor:laguerre}
 The barycentric weights for the roots of the Laguerre polynomial $L_{n+1}^{(\alpha)}(x)$ are
\begin{equation}\label{eq:weights_laguerre}
 \lambda_j^{\mathrm{Lag}}=C^{(\alpha)}(-1)^j\sqrt{x_jw_j}, \qquad j=0,\ldots,n,
\end{equation}
where $C^{(\alpha)}=1$ for the simplified weights, and otherwise
\begin{equation}\label{eq:lag common factor}
 C^{(\alpha)}=\sigma \, \frac{1}{\sqrt{\Gamma(n+\alpha+2)(n+1)!}}.
\end{equation}
\end{corollary}

Finally, the Hermite polynomials satisfy \cite[p.~106]{szego1939polynomials}
\begin{equation*}
H_{n}{''}(x)-2xH_n{'}(x)+2nH_n(x)=0,
\end{equation*}
from which we deduce that
\begin{equation*}
 \phi(x) = 1, \qquad \nu_n = 2n.
\end{equation*}
Moreover, using the fact that
\begin{equation*}
k_{n+1}=2^{n+1}, \quad h_{n+1}=\sqrt{\pi}2^{n+1}(n+1)!,
\end{equation*}
we obtain the following corollary.

\begin{corollary}\label{cor:hermite}
 The barycentric weights for the roots of the Hermite polynomial $H_{n+1}(x)$ are
\begin{equation}\label{eq:weights_hermite}
 \lambda_j^{\mathrm{H}}=C \, (-1)^j \, \sqrt{w_j}, \qquad j=0,\ldots,n,
\end{equation}
where $C=1$ for the simplified weights, and otherwise
\begin{equation*}
 C=\sigma \sqrt{\frac{2^n}{(n+1)!\sqrt{\pi}}}.
\end{equation*}
\end{corollary}

The above corollaries show a close and simple connection between the barycentric weights and the nodes and weights of the corresponding Gaussian quadrature rule for all classical orthogonal polynomials. Since the Legendre polynomials and the Chebyshev polynomials of the first and second kinds are all special cases of Jacobi polynomials, the barycentric weights for the zeros of these polynomials, given in Section~\ref{s:intro} can be derived as immediate consequence.

\section{Additional interpolation points}\label{s:radaulobatto}

A set of Gaussian quadrature points is sometimes augmented with a small set of additional points. Two useful examples are Gauss-Radau and Gauss-Lobatto rules, where one or two (respectively) of the endpoints of the integration interval are added to the set of quadrature points. The weights of the Radau and Lobatto variants can be written in terms of the weights of a regular Gaussian quadrature rule. We show that in some cases one can also write the barycentric weights in terms of the Gauss-Radau or Gauss-Lobatto quadrature weights. This setting covers the set of Chebyshev points of the second kind, i.e., the set of all maxima of the Chebyshev polynomials of the first kind on $[-1,1]$, for which explicit formulae for the barycentric weights are already known.

\subsection{Gaussian quadrature with preassigned abscissae}\label{ss:gauss_preassigned}

We start out in a more general setting in order to illustrate the scope of the arguments. We study Gaussian quadrature rules with a number of preassigned abscissae, for example the set $\{y_j\}_{j=1}^m$. We are interested in a set of $n+1$ additional quadrature points $x_j$ such that the quadrature rule
\begin{equation}\label{eq:rule_pre}
 \int_{-1}^1 w(x) f(x) dx \simeq \sum_{j=1}^m \hat{w}_j f(y_j) + \sum_{j=0}^n \tilde{w}_j f(x_j)
\end{equation}
is exact for polynomials up to degree $m+2n+1$. This is known to be
the maximal order \cite[p.~101]{davis1984numericalintegration}. It is achieved by taking
$x_j$ as the roots of an orthogonal polynomial $\pi_{n+1}$, if it
exists, that is orthogonal in the sense that
\begin{equation}\label{eq:orthogonality_pre}
 \int_{-1}^1 w(x) r_m(x) \pi_{n+1}(x) x^j dx = 0, \qquad j=0,\ldots,n,
\end{equation}
where
\[
 r_m(x) = \prod_{j=1}^m (x-y_j)
\]
is a polynomial of degree $m$ that vanishes at the preassigned quadrature nodes $y_j$. Note that the existence of this polynomial is not guaranteed, even for positive $w(x)$, if $r_m(x)$ switches sign in the integration interval.

Let $w_j$ be the set of weights corresponding to the regular
Gaussian quadrature rule associated with $\pi_{n+1}$. The weights
$\tilde{w}_j$ of the rule with preassigned nodes relate to $w_j$ in
a simple way. As the following statement is not found in the
standard textbook \cite{davis1984numericalintegration}, we include a proof.

\begin{lemma}\label{lem:preassigned}
 We have
\[
 \tilde{w}_j = \frac{w_j}{r_m(x_j)}.
\]
\end{lemma}
\begin{proof}
 Let $f$ be a polynomial of degree $m+2n+1$. Let $f_1$ be the polynomial of degree $m$ that interpolates $f$ at $y_j$,
\[
 f_1(y_j) = f(y_j), \qquad j=1,\ldots,m.
\]
By construction, the remainder $f-f_1$ is divisible by $r_m$ and we define
\[
 f_2(x) = \frac{f(x)-f_1(x)}{r_m(x)},
\]
such that $f(x) = f_1(x) + r_m(x) f_2(x)$. We have
\begin{equation}\label{eq:I_dec}
 \int_{-1}^1 w(x) f(x) dx = \int_{-1}^1 w(x) f_1(x) dx + \int_{-1}^1 w(x) r_m(x) f_2(x) dx.
\end{equation}

Since the first integral in the right hand side of \eqref{eq:I_dec} depends only on $f(y_j)$, and since we are interested only in the weights $\hat{w}_j$ corresponding to $f(x_j)$, we focus on the second integral. Note that $f_2(x)$ has degree $2n+1$ by construction. Thus, the integral can be evaluated exactly with the interpolatory (Gaussian) quadrature rule based on using the roots of $\pi_{n+1}$, which is orthogonal with respect to the weight $w(x) r_m(x)$:
\begin{align*}
 \int_{-1}^1 w(x) r_m(x) f_2(x) dx &= \sum_{j=0}^n w_j f_2(x_j) \\
&= \sum_{j=0}^n w_j \frac{f(x_j) - f_1(x_j)}{r_m(x_j)}.
\end{align*}
Since $f_1(x_j)$ depends only on $f(y_j)$, and not on $f(x_j)$, the result follows.
\end{proof}

In the following, we will explore the corresponding generalizations of our earlier explicit expressions for the barycentric weights. Let $\lambda^x_j$ be the barycentric weights corresponding to the point set $\{x_j\}_{j=0}^n$ and $\lambda^y_j$ those corresponding to the point set $\{y_j\}_{j=1}^m$. For the combined set $\{x_j\}_{j=0}^n \cup \{y_j\}_{j=1}^m$, we denote by $\tilde{\lambda}_{j}$ the barycentric weights corresponding to the points $x_j$ and by $\hat{\lambda}_{j}$ the weights corresponding to the points $y_j$.\footnote{We will consistently use the notation $\tilde{\lambda}$ and $\tilde{w}$ for barycentric weights and quadrature weights corresponding to the roots of an orthogonal polynomial, and $\hat{\lambda}$ and $\hat{w}$ for the added points.}

We derive different expressions for the barycentric weights $\tilde{\lambda}_j$ and $\hat{\lambda}_j$. For the former, from~\eqref{eq:weight} we have
\begin{equation}\label{eq:inner_weights}
 \tilde{\lambda}_{j} = \frac{k_{n+1}}{\pi_{n+1}'(x_j) r_m(x_j)} = \frac{\lambda^x_j}{r_m(x_j)}, \qquad j=0,\ldots,n,
\end{equation}
where $k_{n+1}$ is the leading order coefficient of $\pi_{n+1}(x)$. For the latter, we find
\begin{equation}\label{eq:outer_weights}
 \hat{\lambda}_{j} = \frac{k_{n+1}}{\pi_{n+1}(y_j) r_m'(y_j)} = \lambda^y_j \frac{k_{n+1}}{\pi_{n+1}(y_j)}, \qquad j=1,\ldots,m.
\end{equation}
Assuming $m$ is fixed and $m \ll n$, the latter case presents no computational difficulties. The orthogonal polynomial $\pi_{n+1}$ can be evaluated in ${\mathcal O}(n)$ operations at a single point, for example based on the three-term recurrence relation. Thus, evaluating $\pi_{n+1}(y_j)$ at the $m$ points $y_j$ requires only ${\mathcal O}(mn)$ operations. The barycentric weights $\lambda^y_j$ are easily computed in at most ${\mathcal O}(m^2)$ operations. Hence, we focus on the weights given by~\eqref{eq:inner_weights}. We will consider a number of interesting cases in which roots of classical orthogonal polynomials are supplemented with additional interpolation points.

\subsection{Gauss-Lobatto variant for Jacobi polynomials}\label{ss:jacobi_extrema}

Let us consider first the Gauss-Lobatto points associated with Jacobi polynomials. Thus, we consider the additional points $\{y_j\} = \{-1,1\}$. Since
\[
r_2(x) = (x-1)(x+1) = x^2-1,
\]
we find for a weight of Jacobi type that
\[
 w(x) r_2(x) = (1-x)^\alpha (1+x)^\beta (x^2-1) = -(1-x)^{\alpha+1} (1+x)^{\beta+1}.
\]
For notational convenience, we let $w_j$ denote the Gaussian quadrature weights with respect to the positive weight function $(1-x)^{\alpha+1} (1+x)^{\beta+1}$, while $\tilde{w}_j$ denotes the corresponding interior weights of the Gauss-Lobatto quadrature rule for the Jacobi weight function $w(x)=(1-x)^\alpha (1+x)^\beta$. Hence, both $w_j$ and $\tilde{w}_j$ are positive values. With this notation, the result of Lemma~\ref{lem:preassigned} should be modified so that
\begin{equation}\label{eq:interior Lobatto quadweights}
 \tilde{w}_j = - \frac{w_j}{r_2(x_j)} = \frac{w_j}{1-x_j^2}, \quad j=0,\ldots,n.
\end{equation}

It is easy to see that the interior nodes for Gauss-Lobatto
integration with respect to $w(x)$ are the roots of
$P_{n+1}^{(\alpha+1,\beta+1)}(x)$. These are precisely the roots of
$P_{n+2}^{(\alpha,\beta)}{'}(x)$, or the extrema of
$P_{n+2}^{(\alpha,\beta)}(x)$ in $(-1,1)$.

\begin{theorem}\label{th:jacobi_lobatto}
 Let $x_j$ be the roots of $P_{n+1}^{(\alpha+1,\beta+1)}(x)$ and denote by $\tilde{w}_j$ the corresponding interior weights of the Gauss-Lobatto quadrature rule for the Jacobi weight function $w(x)=(1-x)^\alpha (1+x)^\beta$. Then we may choose
\begin{equation}\label{eq:lobatto_w}
 \tilde{\lambda}_{j} =  C^{(\alpha+1,\beta+1)}_n (-1)^{j+1} \sqrt{\tilde{w}_j}, \qquad j=0,\ldots,n,
\end{equation}
with $ C^{(\alpha+1,\beta+1)}_n $ defined as in \eqref{eq:jacobi
common factor}.

The corresponding barycentric weights for the points $\pm 1$ are
\begin{equation}\label{eq:lobatto_w1}
 \hat{\lambda}_{1} =  C^{(\alpha+1,\beta+1)}_n \, \sqrt{(\beta+1) \hat{w}_{1}},
\end{equation}
and
\begin{equation}\label{eq:lobatto_w2}
\hat{\lambda}_{2} =  C^{(\alpha+1,\beta+1)}_n \, (-1)^{n} \,
\sqrt{(\alpha+1) \hat{w}_{2}}.
\end{equation}
\end{theorem}
\begin{proof}
 Let $\lambda^x_j$ be the barycentric weights corresponding to the point set $\{x_j\}_{j=0}^{n}$. From Corollary~\ref{cor:jacobi} we already know that
\[
 \lambda^x_j =   C^{(\alpha+1,\beta+1)}_n \, (-1)^j \, \sqrt{(1-x_j^2)w_j}.
\]
Combining this expression with Lemma~\ref{lem:preassigned} leads to
\[
 \tilde{ \lambda}_{j} = \frac{\lambda^x_j}{-(1-x_j^2)} = -  C^{(\alpha+1,\beta+1)}_n \frac{(-1)^j \sqrt{(1-x_j^2)w_j}}{1-x_j^2} =   C^{(\alpha+1,\beta+1)}_n (-1)^{j+1} \sqrt{\tilde{w}_j}.
\]
This shows~\eqref{eq:lobatto_w}.

It remains to determine the barycentric weights corresponding to the
endpoints. Let $k_{n+1}^{(\alpha+1,\beta+1)}$ denotes the leading
coefficients of the Jacobi polynomial
$P_{n+1}^{(\alpha+1,\beta+1)}(x)$. It is known that (see \cite[p.~59
and p.~63]{szego1939polynomials})
\[
k_{n+1}^{(\alpha+1,\beta+1)}=\frac{1}{2^{n+1}}\binom{2n+\alpha+\beta+4}{n+1}, \quad P_{n+1}^{(\alpha+1,\beta+1)}(-1) = (-1)^{n+1} \binom{n+\beta+2}{n+1}.
\]
We have from~\eqref{eq:outer_weights} that
\begin{align}
 \hat{\lambda}_{1} &=  \frac{k_{n+1}^{(\alpha+1,\beta+1)}}{r_2'(-1) P_{n+1}^{(\alpha+1,\beta+1)}(-1)}\nonumber \\
 & = -\frac12 k_{n+1}^{(\alpha+1,\beta+1)}
 \frac{1}{P_{n+1}^{(\alpha+1,\beta+1)}(-1)}\nonumber \\
 & =
 \frac{(-1)^{n}}{2^{n+2}}\frac{\Gamma(2n+\alpha+\beta+5)\Gamma(\beta+2)}{\Gamma(n+\beta+3)\Gamma(n+\alpha+\beta+4)}.
\end{align}
We have used $r_2(x) = x^2 - 1$, so that $r_2'(-1)=-2$. From
\cite[Eqs. (3.10) and (3.11)]{gautschi2000lobatto} we know that
\begin{align*}
\hat{w}_1 &=
2^{\alpha+\beta+1}\frac{\Gamma(\alpha+2)\Gamma(\beta+1)}{\Gamma(\alpha+\beta+3)}\frac{\binom{n+\alpha+2}{n+1}}{\binom{n+\beta+2}{n+1}\binom{n+\alpha+\beta+3}{n+1}}\nonumber
\\
&=2^{\alpha+\beta+1}
\frac{\Gamma(\beta+1)\Gamma(n+\alpha+3)\Gamma(\beta+2)(n+1)!}{\Gamma(n+\beta+3)\Gamma(n+\alpha+\beta+4)},
\end{align*}
and
\begin{align*}
\hat{w}_2 &=
2^{\alpha+\beta+1}\frac{\Gamma(\beta+2)\Gamma(\alpha+1)}{\Gamma(\alpha+\beta+3)}\frac{\binom{n+\beta+2}{n+1}}{\binom{n+\alpha+2}{n+1}\binom{n+\alpha+\beta+3}{n+1}}\nonumber
\\
&=2^{\alpha+\beta+1}
\frac{\Gamma(\alpha+1)\Gamma(n+\beta+3)\Gamma(\alpha+2)(n+1)!}{\Gamma(n+\alpha+3)\Gamma(n+\alpha+\beta+4)}.
\end{align*}
Motivated by the form of \eqref{eq:lobatto_w}, after some
calculations, we find that
\begin{align*}
( C^{(\alpha+1,\beta+1)}_n )^2 \hat{w}_1&= \frac{\Gamma(2n+\alpha+\beta+5)^2\Gamma(\beta+1)\Gamma(\beta+2)}{2^{2n+4}\Gamma(n+\beta+3)^2\Gamma(n+\alpha+\beta+4)^2}\nonumber \\
&=\frac{1}{\beta+1}\left(\frac{\Gamma(2n+\alpha+\beta+5)\Gamma(\beta+2)}{2^{n+2}\Gamma(n+\beta+3)\Gamma(n+\alpha+\beta+4)}\right)^2
\nonumber \\
&=\frac{\hat{\lambda}_{1}^2}{\beta+1}.
\end{align*}
Hence,
\begin{align*}
\hat{\lambda}_{1}^2 = ( C^{(\alpha+1,\beta+1)}_n )^2 (\beta + 1 )
\hat{w}_1.
\end{align*}
Similarly, we can show that
\begin{align*}
\hat{\lambda}_{2}^2 = ( C^{(\alpha+1,\beta+1)}_n )^2 (\alpha + 1 )
\hat{w}_2.
\end{align*}
Again noting that the barycentric weights have alternating signs,
expressions ~\eqref{eq:lobatto_w1} and \eqref{eq:lobatto_w2}
follow.

\end{proof}

Note for completeness that in Theorem~\ref{th:jacobi_lobatto} we consider interpolation in a set of $n+3$ points in total. These points are
\[
 \{ -1,1 \} \cup \{ x_j \}_{j=0}^n,
\]
where $x_j$ are the $n+1$ roots of $P_{n+1}^{(\alpha+1,\beta+1)}$. In our current notation, the weights of the corresponding Gauss-Jacobi quadrature rule, relative to the weight function $(1-x)^\alpha (1+x)^\beta$, are
\[
 \{ \hat{w}_{1}, \hat{w}_{2} \} \cup \{ \tilde{w}_{j} \}_{j=0}^n.
\]

The result of Theorem~\ref{th:jacobi_lobatto} may be written more
concisely as follows.
\begin{corollary}\label{cor:jacobi_lobatto}
 Let $-1 = x_0 < x_1 < \cdots < x_n = 1 $ be the roots of $(1-x^2) P_{n}^{(\alpha,\beta)}{'}(x)$ and let $\{ w_j \}_{j=0}^{n}$ be
 the corresponding weights of the interpolatory quadrature rule associated with the weight function $(1-x)^\alpha(1+x)^\beta$. Then for $n \geq 1$, the barycentric
 weights for the interpolation points $\{ x_j \}_{j = 0}^{n}$ are
\begin{align}
 \lambda_j =  C^{(\alpha+1,\beta+1)}_{n - 2}  (-1)^j \sqrt{\delta_j w_j}, \quad
\delta_j = \left\{\begin{array}{ll}
                                          \beta+1,    & \hbox{$\textstyle  j=0$},\\
                                          \alpha+1,   & \hbox{$\textstyle  j=n$},\\
                                          1,          & \hbox{otherwise}.
                                        \end{array}\right.
\end{align}
The simplified barycentric weights can be obtained directly by
deleting the factor $C^{(\alpha+1,\beta+1)}_{n - 2} $.
\end{corollary}

We are especially concerned with some special cases of the Gauss-Jacobi-Lobatto points. When $\alpha=\beta=-1/2$, this corresponds to the Gauss-Chebyshev-Lobatto points, which are also called Chebyshev points of the second kind or Clenshaw-Curtis points. The Gauss-Chebyshev-Lobatto quadrature rule is given by
\begin{displaymath}
\int_{-1}^{1}\frac{f(x)}{\sqrt{1-x^2}}dx\simeq\frac{\pi}{n}\sum_{k=0}^{n}{''}f(x_k),
\end{displaymath}
where the double prime denotes a sum whose first and last terms are
halved and the Gauss-Chebyshev-Lobatto points $x_k$ are given
explicitly in \eqref{eq:points_secondkind}. The following corollary is an immediate
consequence of Corollary \ref{cor:jacobi_lobatto}.

\begin{corollary}
For Gauss-Chebyshev-Lobatto points, the simplified barycentric
weights are given by
\begin{displaymath}
\lambda_j^{\mathrm{CH2}}=(-1)^j\delta_j,\quad
\delta_j=\bigg\{\begin{array}{cc}
                                          1/2,   & \mbox{$\textstyle  j=0$ or $j=n$},\\
                                          1,   & \mbox{otherwise}.
                                        \end{array}
\end{displaymath}
\end{corollary}

Thus, we have provided an alternative simple derivation of the barycentric weights for the Chebyshev points of the second kind.

Another important case of $\alpha=\beta=0$ corresponds to the Gauss-Legendre-Lobatto points. The Gauss-Legendre-Lobatto quadrature rule is defined by
\begin{displaymath}
\int_{-1}^{1}f(x)dx\simeq\sum_{j=0}^{n}w_jf(x_j),
\end{displaymath}
where the Gauss-Legendre-Lobatto points $\{x_j\}_{j=0}^{n}$ are the zeros of $(1-x^2)P_n'(x)$ and $P_n(x)$ is the Legendre polynomial of degree $n$. The following corollary gives the simplified barycentric weights for the Gauss-Legendre-Lobatto points.
\begin{corollary}\label{eq:gausslegendrelobatto}
For Gauss-Legendre-Lobatto points, the simplified barycentric weights $\lambda_j^{\mathrm{GLL}}$ are given by
\begin{equation}
\lambda_j^{\mathrm{GLL}} = (-1)^j \sqrt{w_j}, \qquad j = 0,\ldots,n,
\end{equation}
where $w_j$ are the Gauss-Legendre-Lobatto quadrature weights.
\end{corollary}
\begin{proof}
It follows readily from Corollary \ref{cor:jacobi_lobatto}.
\end{proof}

Below, we list the steps for computing the interpolant that
interpolates $f(x)$ at the Gauss-Jacobi-Lobatto points, i.e. roots
of $(1-x^2)P_{n}^{(\alpha,\beta)}{'}(x)$.

\vspace{.2cm} {\sc Algorithm 1.} Computation of the
Gauss-Jacobi-Lobatto interpolant:
\begin{enumerate}
\item  Compute the nodes and weights of the $(n-1)$-point
Gauss-Jacobi quadrature with respect to the weight
$(1-x)^{\alpha+1}(1+x)^{\beta+1}$ by the Hale-Townsend
algorithm \cite{hale2012fastgauss}. The interpolation points are the $(n-1)$ nodes
supplemented with two additional points $\pm 1$.
\item Calculate the interior Gauss-Jacobi-Lobatto quadrature weights by
\eqref{eq:interior Lobatto quadweights} and compute the two boundary
quadrature weights by their explicit expressions.
\item Evaluate the barycentric weights by the corollary \ref{cor:jacobi_lobatto}.
\item Compute the Gauss-Jacobi-Lobatto interpolant by its
barycentric representation.
\end{enumerate}

In step 2, the two boundary quadrature weights can be computed
directly if $n$ is smaller than about 100. When $n$ is larger than
about 100, however, direct evaluation of both boundary quadrature
weights would result in an overflow. This problem can be avoided by
reformulating both boundary quadrature weights via logarithms.

\subsection{Gauss-Radau variant for Jacobi polynomials}

The Gauss-Radau variant is one where we include only the left endpoint $x=-1$. In that case, we have
\[
 r_1(x) = x+1.
\]
Thus, the other quadrature points are the roots of $P_{n+1}^{(\alpha,\beta+1)}$. We have the following result.

\begin{theorem}\label{th:jacobi radau}
Let $\{x_j\}_{j=0}^{n}$ be the roots of $P_{n+1}^{(\alpha,\beta+1)}(x)$ and denote by $\tilde{w}_j$ the corresponding
weights of the Gauss-Jacobi-Radau quadrature rule with respect to the Jacobi weight function $w(x)=(1-x)^\alpha (1+x)^\beta$. Then the barycentric weights corresponding to the interior nodes $x_j$ are given by
\begin{align}\label{eq:inner jacobi radau}
\tilde{\lambda}_{j} = C^{(\alpha,\beta+1)}_n (-1)^j \sqrt{( 1 - x_j
) \tilde{w}_j},\quad   j = 0,\ldots,n.
\end{align}
The barycentric weight corresponding to the point $x=-1$ is
\begin{align}\label{eq:outer jacobi radau}
\hat{\lambda}_{1} = - C^{(\alpha,\beta+1)}_n \sqrt{2(\beta+1)
\hat{w}_1}.
\end{align}
\end{theorem}
\begin{proof}
 Let $\lambda_{j}^{x}$ be the barycentric weights corresponding to the point set
 $\{x_j\}_{j=0}^{n}$. From \eqref{eq:inner_weights} and Corollary \ref{cor:jacobi}, it follows that
 \begin{align}
\tilde{\lambda}_{j} = \frac{\lambda_{j}^{x}}{1 + x_j} = \frac{
C^{(\alpha,\beta+1) }_n (-1)^j \sqrt{(1 - x _j^2 ) w_j}}{ 1 + x_j }
= C^{(\alpha,\beta+1)}_n (-1)^j \sqrt{(1 - x_j) \tilde{w}_j}.
 \end{align}
 This proves \eqref{eq:inner jacobi radau}. For the barycentric
 weight corresponding to the point $x=-1$, applying
 \eqref{eq:outer_weights} yields
\begin{align}
\hat{\lambda}_{1}&=\frac{k_{n+1}^{(\alpha,\beta+1)}}{r_1'(-1)
P_{n+1}^{(\alpha,\beta+1)}(-1)}\nonumber \\
&=\frac{(-1)^{n+1}\Gamma(2n+\alpha+\beta+4)\Gamma(\beta+2)}{2^{n+1}\Gamma(n+\beta+3)\Gamma(n+\alpha+\beta+3)}.
\end{align}
On the other hand, from \cite{gautschi2000radau} we have
\begin{align*}
\hat{w}_1&=\frac{2^{\alpha+\beta+1}\Gamma(\beta+1)\Gamma(n+\alpha+2)}{\binom{n+\beta+2}{n+1}\Gamma(n+\alpha+\beta+3)}\\
&=\frac{2^{\alpha+\beta+1}\Gamma(\beta+1)\Gamma(\beta+2)\Gamma(n+\alpha+2)\Gamma(n+2)}{\Gamma(n+\beta+3)\Gamma(n+\alpha+\beta+3)}.
\end{align*}
and by some computations,
\begin{align*}
(C^{(\alpha,\beta+1)}_n)^2 \hat{w}_1&=\frac{\Gamma(\beta+1)\Gamma(\beta+2)\Gamma(2n+\alpha+\beta+4)^2}{2^{2n+3}\Gamma(n+\alpha+\beta+3)^2\Gamma(n+\beta+3)^2} \nonumber\\
&=\frac{1}{2(\beta+1)}\frac{\Gamma(\beta+2)^2\Gamma(2n+\alpha+\beta+4)^2}{2^{2n+2}\Gamma(n+\alpha+\beta+3)^2\Gamma(n+\beta+3)^2}  \nonumber \\
&=\frac{\hat{\lambda}_{1}^2}{2(\beta+1)}.
\end{align*}
Thus we find
\begin{align*}
\hat{\lambda}_{1}^2=2(\beta+1)(C^{(\alpha,\beta+1)}_n)^2 \hat{w}_1.
\end{align*}
Since the barycentric weights have alternating signs, expression
\eqref{eq:outer jacobi radau} follows.
\end{proof}

Note for completeness that in the last theorem we consider barycentric interpolation in a set of $n+2$ points in total. These points are
\[
 \{ -1 \} \cup \{ x_j \}_{j=0}^n,
\]
where $x_j$ are the $n+1$ roots of $P_{n+1}^{(\alpha,\beta+1)}$. In our current notation, the corresponding Gauss-Radau quadrature weights are
\[
 \{ \hat{w}_{1} \} \cup \{ \tilde{w}_{j} \}_{j=0}^n.
\]

The result of Theorem~\ref{th:jacobi radau} may be written more
concisely as follows.
\begin{corollary}\label{cor:jacobi_radau}
 Let $-1=x_0<x_1<\cdots<x_n<1$ be the roots of $(1+x) P_{n}^{(\alpha,\beta+1)}(x)$ and let $w_j$ be the corresponding weights of the interpolatory
 quadrature rule with the weight function $(1-x)^\alpha(1+x)^\beta$. Then the simplified barycentric weights are
\[
 \lambda_j = (-1)^j \sqrt{(1-x_j)\delta_j w_j},\quad
\delta_j=\left\{\begin{array}{cc}
                                          \beta+1,    & \hbox{$\textstyle  j=0$},\\
                                          1,          & \hbox{otherwise}.
                                        \end{array}\right.
\]
\end{corollary}

Similar results hold if one chooses to add the other endpoint $x=+1$ instead.

We remark that the steps for computing the Gauss-Jacobi-Radau
interpolant are similar to the Lobatto case. We omit the details.

\subsection{Gauss-Radau variant for Laguerre polynomials}\label{ss:laguerre_radau}

Finally, we consider a Radau variant for Laguerre polynomials. We
include the left endpoint $x=0$ of the half-infinite integration
interval $[0,\infty)$ as a pre-assigned quadrature point and thus we
have $r_1(x)=x$. The result is the following.

\begin{theorem}\label{th:laguerre_radau}
Let $x_j$ be the roots of $L_{n+1}^{(\alpha+1)}(x)$ and denote by $\tilde{w}_j$ the corresponding weights of the Gauss-Laguerre-Radau quadrature rule with respect to the Laguerre weight function $w(x)=x^{\alpha} e^{-x}$. Then the barycentric weights corresponding to the interior nodes $x_j$ are given by
\begin{equation}\label{eq:weights_laguerre_radau}
\tilde{\lambda}_{j} = C^{(\alpha+1)}(-1)^j\sqrt{\tilde{w}_j},\quad
j=0,\ldots,n.
\end{equation}
The barycentric weight corresponding to the point $x=0$ is
\[
\hat{\lambda}_1 = -C^{(\alpha+1)}\sqrt{(\alpha+1) \hat{w}_1},
\]
where $C^{(\alpha+1)}$ is defined as in \eqref{eq:lag common
factor}.
\end{theorem}
\begin{proof}
Let $\lambda_{j}^{x}$ be the barycentric weights corresponding to
the point set $\{x_j\}_{j=0}^{n}$. By virtue of
\eqref{eq:inner_weights} and Corollary \ref{cor:laguerre} yields
\begin{align*}
\tilde{\lambda}_{j}=\frac{\lambda_{j}^{x}}{x_j}=C^{(\alpha+1)}(-1)^j\sqrt{\frac{w_j}{x_j}}=C^{(\alpha+1)}(-1)^j\sqrt{\tilde{w}_j}.
\end{align*}
Let $k_{n+1}$ denote the leading coefficient of the Laguerre
polynomial $L_{n+1}^{(\alpha+1)}(x)$. For the barycentric weight
corresponds to the point $x=0$, using \eqref{eq:outer_weights} we
have that
\begin{align}
\hat{\lambda}_1=\frac{k_{n+1}}{L_{n+1}^{(\alpha+1)}(0)}=(-1)^{n+1}\frac{\Gamma(\alpha+2)}{\Gamma(n+\alpha+3)}.
\end{align}
From \cite[Eq. (6.5)]{gautschi2000radau} we have
\begin{align}
\hat{w}_1=\frac{\Gamma(\alpha+1)}{\binom{n+\alpha+2}{n+1}}=\frac{\Gamma(\alpha+1)\Gamma(\alpha+2)\Gamma(n+2)}{\Gamma(n+\alpha+3)},
\end{align}
and hence, by direct computation,
\begin{align}
(C^{(\alpha+1)})^2\hat{w}_1&=\frac{\Gamma(\alpha+1)\Gamma(\alpha+2)}{\Gamma(n+\alpha+3)^2}\nonumber
\\
&=\frac{1}{\alpha+1}\left(\frac{\Gamma(\alpha+2)}{\Gamma(n+\alpha+3)}\right)^2\nonumber
\\
&=\frac{1}{\alpha+1}\hat{\lambda}_1^2.
\end{align}
Equivalently,
\begin{align*}
\hat{\lambda}_1^2=(\alpha+1)(C^{(\alpha+1)})^2\hat{w}_1.
\end{align*}
Noting that barycentric weights have alternating signs, we obtain
the result.
\end{proof}

%
%
%

\section{Numerical examples}\label{s:examples}

In this section we shall show several numerical examples to
illustrate the performance of the barycentric interpolation formula.
All computations were performed in Matlab in double precision
arithmetic.

\begin{figure}[t]
\centering
\includegraphics[width=7cm]{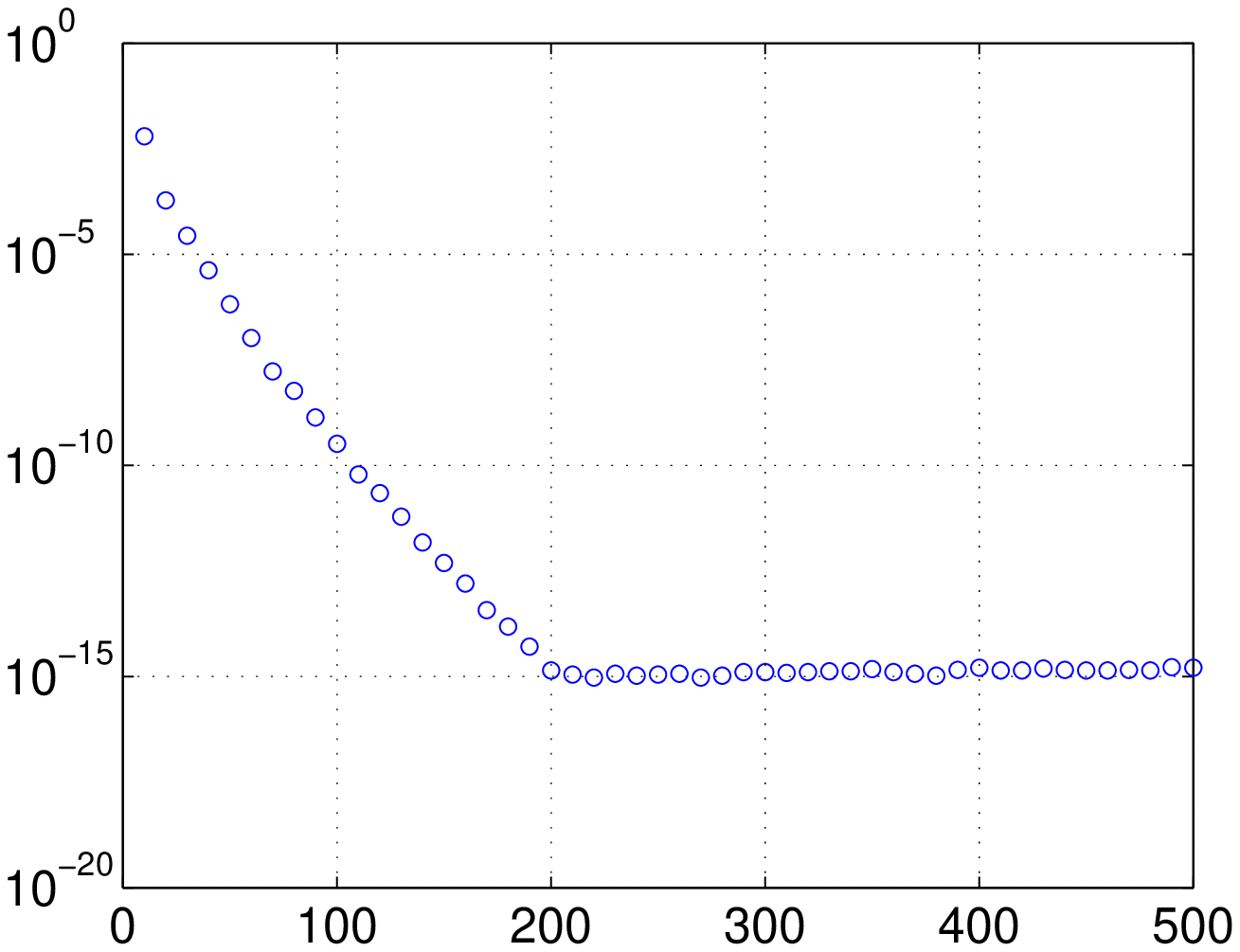}
\quad
\includegraphics[width=7cm]{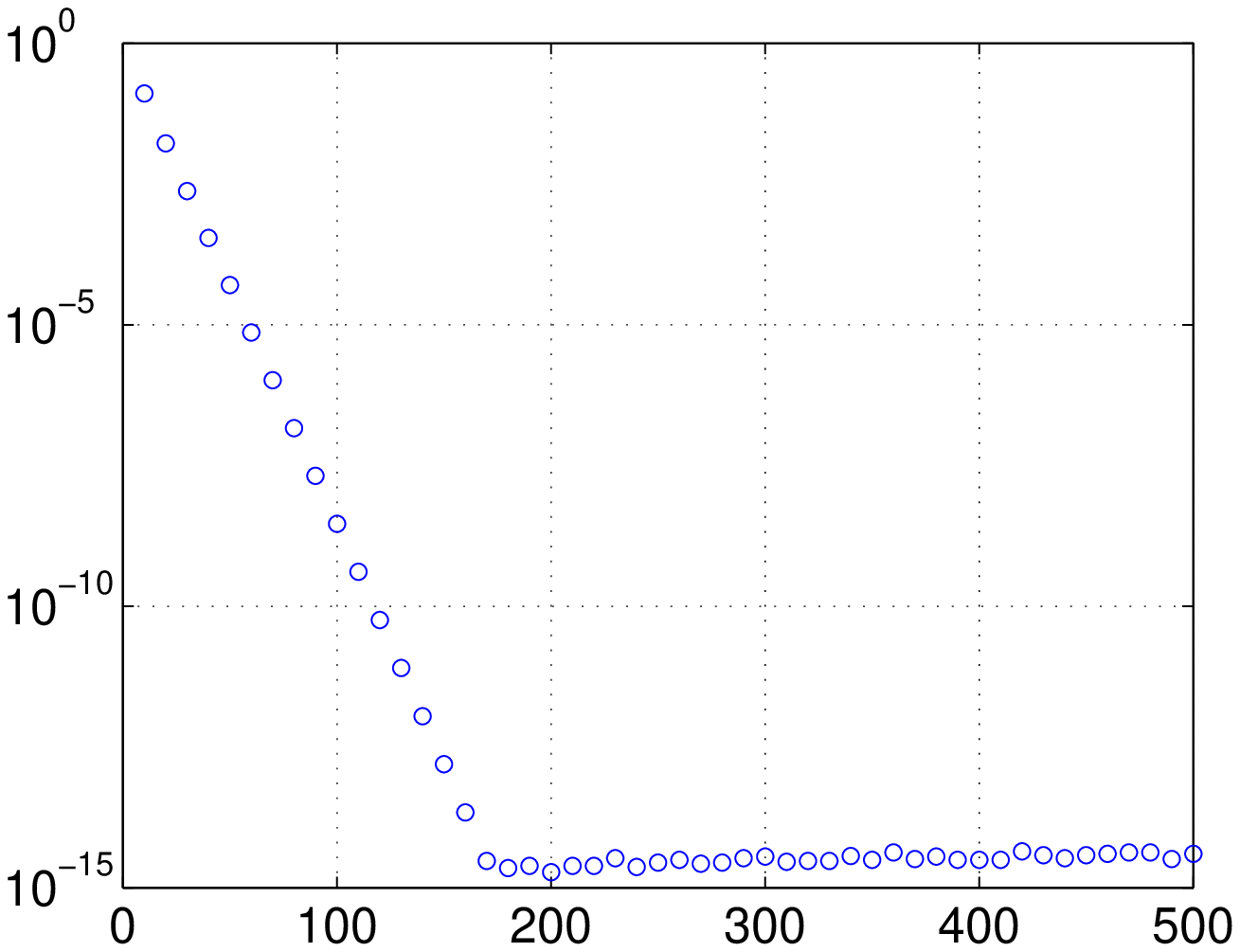}\\
\caption{Convergence of the barycentric Jacobi formula to the
function $f(x)=\frac{1}{1+25x^2}$ (left) and
$f(x)=e^{-\frac{1}{x^2}}$ (right). Here we choose
$\alpha=-\frac{1}{2}$, $\beta=-\frac{1}{4}$ and $n$ ranges from 10
to 500.}\label{firstfig}
\end{figure}

\begin{figure}[t]
\centering
\includegraphics[width=7cm]{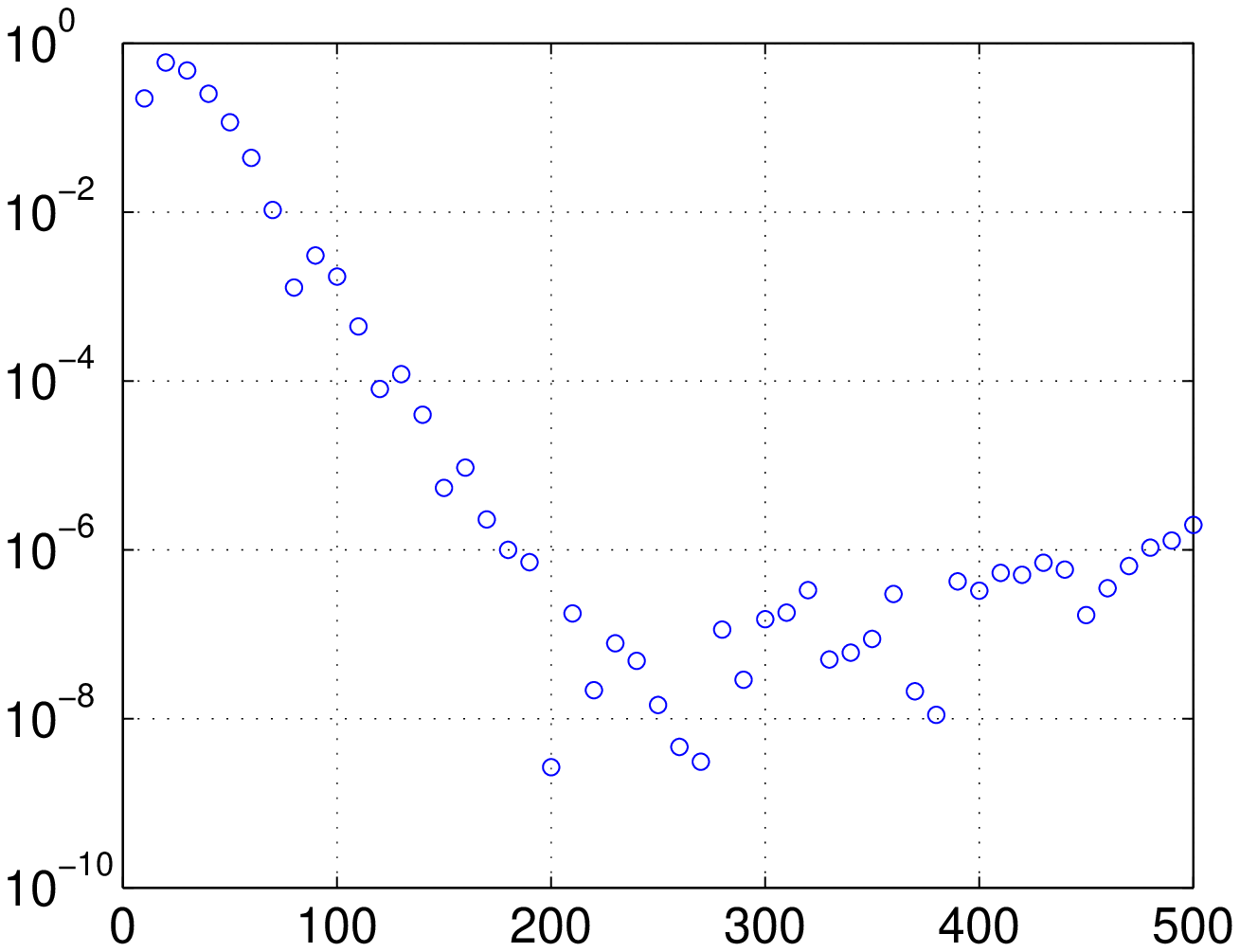}
\quad
\includegraphics[width=7cm]{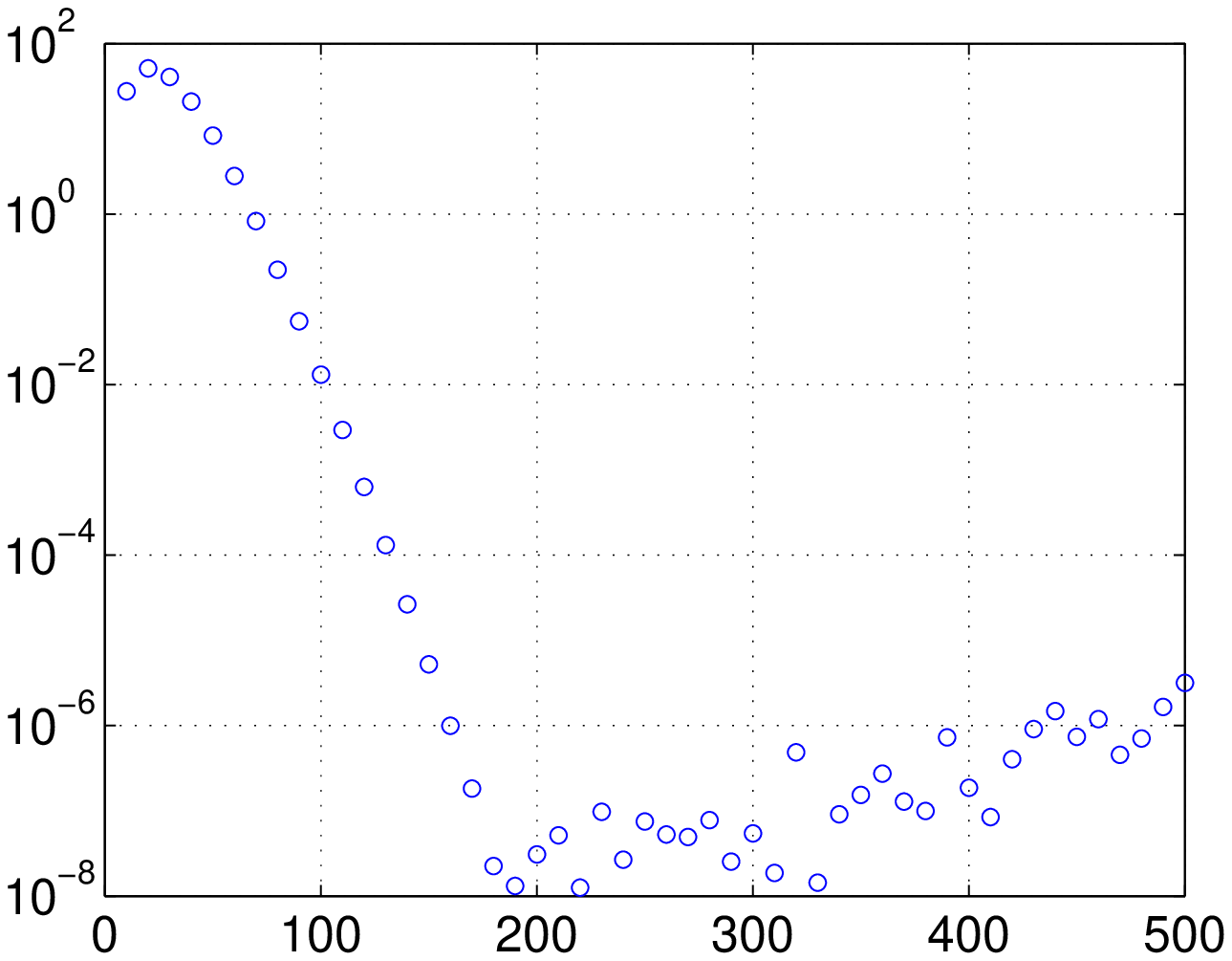}\\
\caption{Convergence of the barycentric Jacobi formula to the two
functions $f(x)=\frac{1}{1+25x^2}$ (left) and
$f(x)=e^{-\frac{1}{x^2}}$ (right). Here, we choose $\alpha=5$,
$\beta=5$ and $n$ ranges from $10$ to $500$.}\label{secondfig}
\end{figure}

\begin{example}
We  first consider the convergence of the barycentric Jacobi
interpolation formula to the two smooth functions
$f(x)=\frac{1}{1+25x^2}$ and $f(x)=e^{-\frac{1}{x^2}}$.
\end{example}
The maximal pointwise error of the barycentric Jacobi formula
\[ \max_{-1\leq x\leq 1}|f(x)-p_n(x)|,    \]
is estimated by measuring at a large number of equispaced points in $[-1,1]$. The nodes and
weights of the Gauss-Jacobi quadrature rule are computed with the Glaser-Liu-Rokhlin algorithm in ${\mathcal O}(n)$ operations. This computation can be performed in the Matlab package Chebfun with the command \texttt{jacpts} \cite{hale2012chebfunquadrature}.
Starting with Chebfun version 4.0, this routine returns the
simplified barycentric weights as well, using
formula~\eqref{eq:weights_jacobi}\footnote{The simplified
barycentric weights returned by the command $\texttt{jacpts}$,
$\texttt{lagpts}$, $\texttt{hermpts}$ are normalized by setting
their maximum value to one. }. Future versions of this routine are likely to be based on the faster Hale-Townsend algorithm \cite{hale2012fastgauss}, but that will not change the asymptotic computational complexity of the experiment.


Figure \ref{firstfig} shows the convergence of the
barycentric Jacobi formula with $\alpha=-\frac{1}{2}$ and
$\beta=-\frac{1}{4}$. We can see that the barycentric Jacobi formula
leads to stable computations. For large $\alpha$ and $\beta$, the Lebesgue constant for Jacobi points becomes very large,
typically $\mathcal{O}(n^{\max\{\alpha,\beta\}+\frac{1}{2}})$
\cite[p.~338]{szego1939polynomials}. Hence, the barycentric Jacobi formula will be
unstable. Figure \ref{secondfig} shows the convergence of the
barycentric Jacobi formula for the same two functions with
$\alpha=\beta=5$. We can see that the barycentric Jacobi formula is
indeed unstable for large $n$, confirming the stability analysis of the
barycentric formula by Higham in \cite{higham2004barycentric}.

\begin{example}
Next, we consider the application of the barycentric Jacobi interpolation
formula to the function $J_{\frac{1}{2}}(x)$ on the interval
$[0,1]$, where $J_{\frac{1}{2}}(x)$ is the Bessel function of the
first kind of order $\frac{1}{2}$. Since the function $J_{\frac{1}{2}}(x)$ behaves like $\sqrt{x}$ when $x\rightarrow0$, we interpolate the function
\[
f(x)=\frac{J_{\frac{1}{2}}(x)}{\sqrt{x}}, \quad x\in[0,1].
\]
We apply the following norm to measure the error of the
barycentric interpolation formula:
\[ \int_{0}^{1}\sqrt{x}|f(x)-p_n(x)|^2dx. \]
\end{example}

\begin{figure}[t]
\centering
\includegraphics[width=7.2cm]{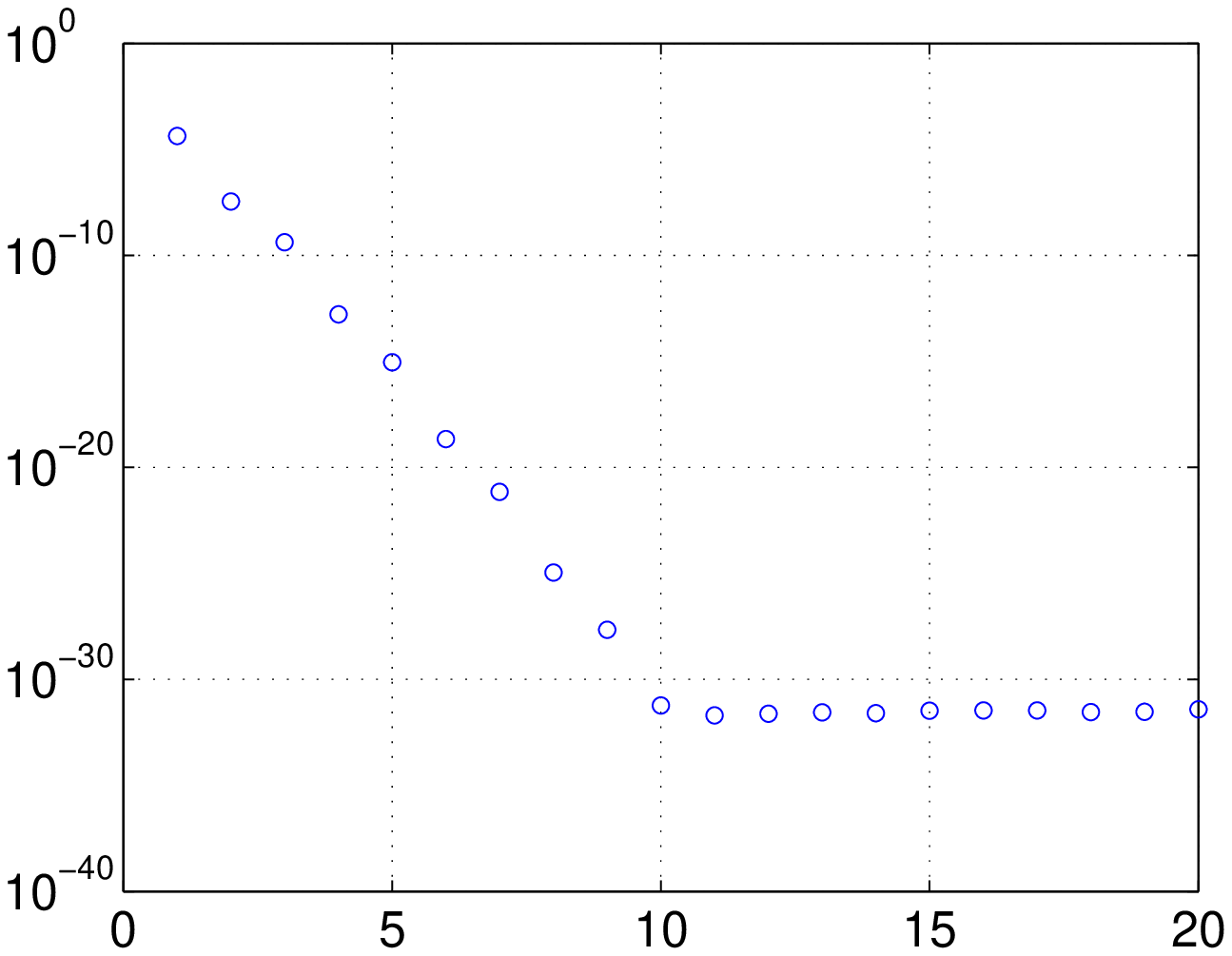}
\quad
\includegraphics[width=7.2cm]{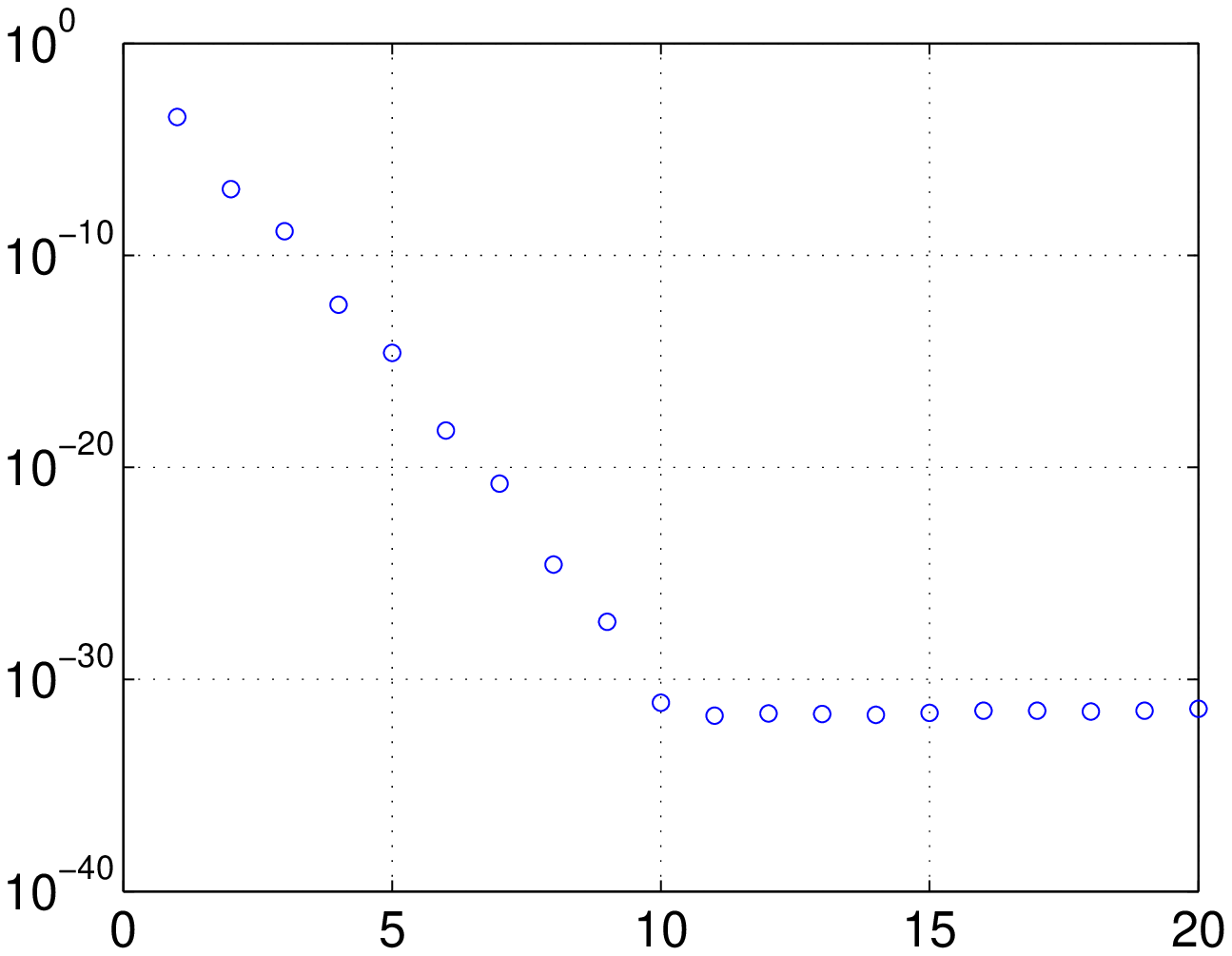}\\
 \caption{Convergence of the barycentric Jacobi (left) and Jacobi Lobatto
 formulas (right). Here $n$ ranges from $1$ to $20$.}\label{thirdfig}
\end{figure}

It is easy to see that this example corresponds to $\alpha=0$ and
$\beta=\frac{1}{2}$. We have applied the barycentric Jacobi and
Jacobi Lobatto formulae to approximate the function $f(x)$. The
barycentric Jacobi weights are computed by \eqref{eq:weights_jacobi}
with $C^{(\alpha,\beta)}_n = 1$. The barycentric
Gauss-Jacobi-Lobatto formula is computed by using the algorithm I
and we have used the simplified barycentric Gauss-Jacobi-Lobatto
weights in our implementation. Numerical results are illustrated in
Figure \ref{thirdfig}.


\begin{example}
Finally, we consider the application of the barycentric Laguerre
interpolation formula to the function
$\mathrm{Ai}((\frac{3}{2}(x+1))^{\frac{2}{3}})$ on the interval
$[0,\infty)$, where $\mathrm{Ai}(x)$ denotes the Airy function.
Since this function behaves like $e^{-x}$ when $x\rightarrow\infty$, we interpolate the function
\begin{equation}\label{eq:airy example}
f(x)=\mathrm{Ai}\left((\frac{3}{2}(x+1))^{\frac{2}{3}}\right)e^{x},
\quad x\in[0,\infty).
\end{equation}
We apply the following norm to measure the error of the
barycentric Laguerre interpolation formula:
\[ \int_{0}^{\infty}e^{-x}|f(x)-p_n(x)|dx. \]
\end{example}

\begin{figure}[t]
\centering
\includegraphics[width=7cm]{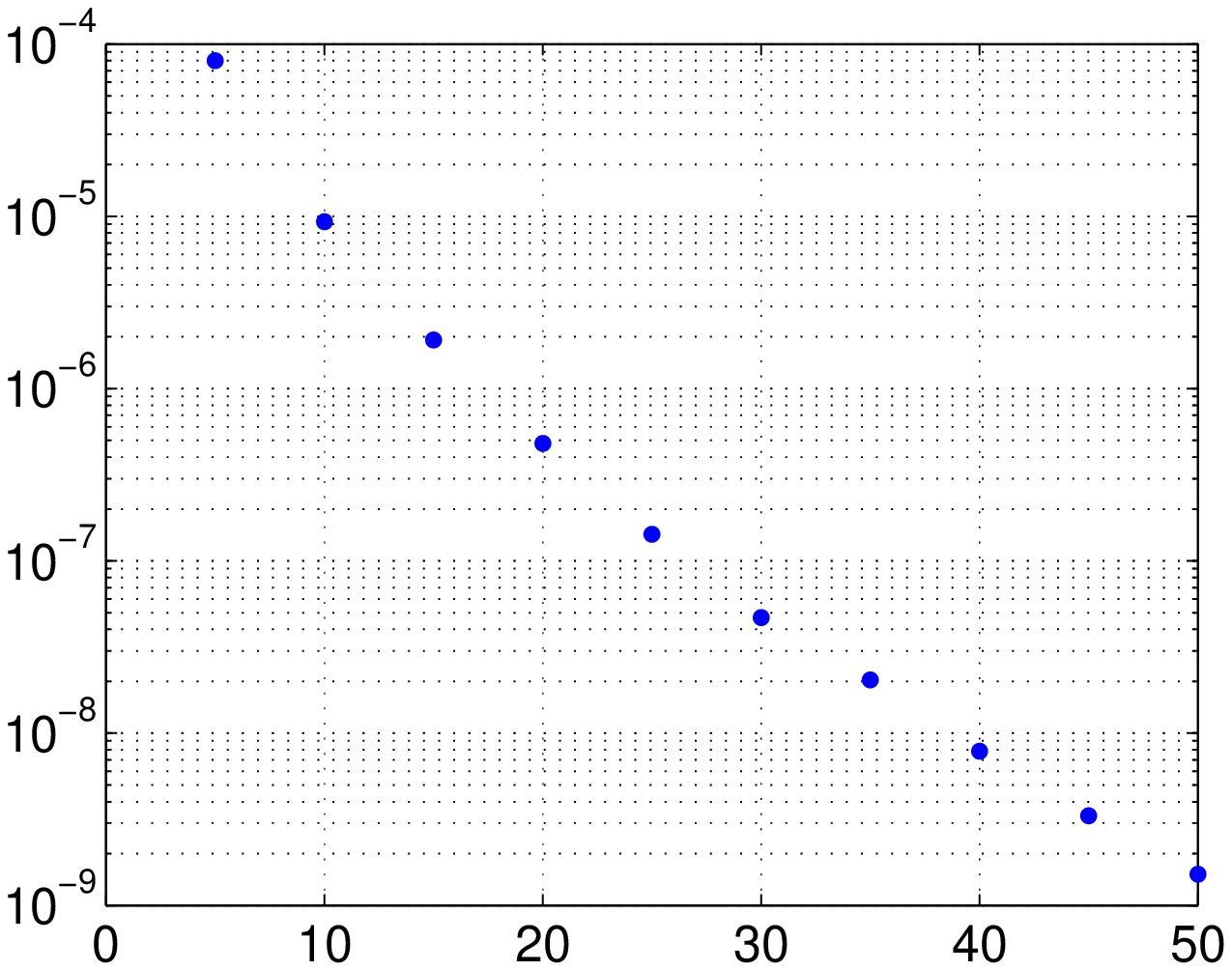}
\quad
\includegraphics[width=7cm]{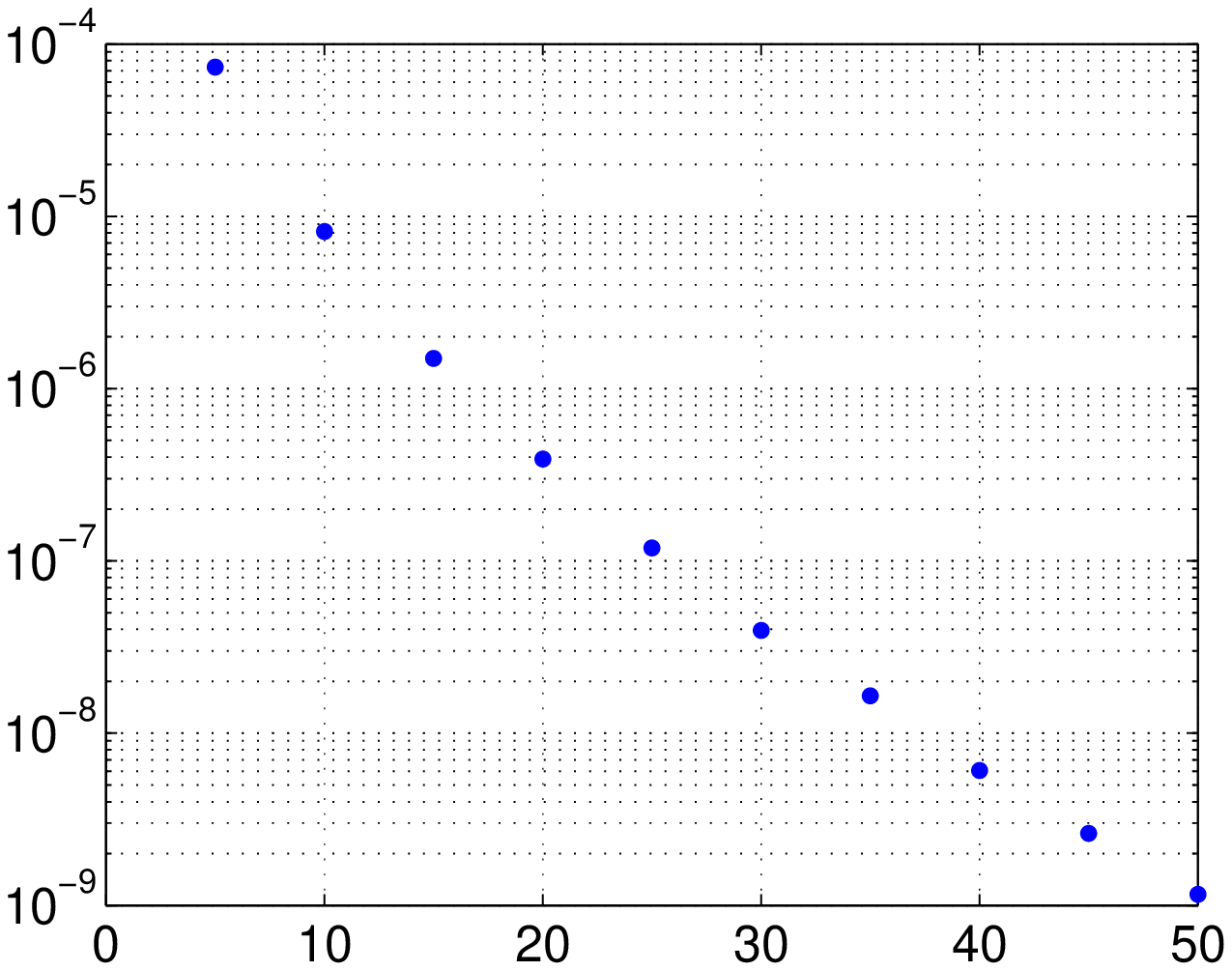}
\caption{Convergence of the barycentric Laguerre (left) and Laguerre
Radau (right) formulae to the function \eqref{eq:airy
example}.}\label{fourthfig}
\end{figure}

For the barycentric Laguerre formula, the nodes and weights of
Gauss-Laguerre quadrature are evaluated in Chebfun with the command
\texttt{lagpts}, which also returns the simplified barycentric
weights using the formula \eqref{eq:weights_laguerre}. For the
barycentric Gauss-Laguerre-Radau formula, the nodes and weights are
evaluated with the Golub-Welsch algorithm, which is based on
computing eigenvalues and eigenvectors of a symmetric tridiagonal
matrix whose elements are obtained from the three-term recurrence
relation satisfied by the Laguerre polynomials \cite{golub1969gauss}. 
The barycentric Gauss-Laguerre-Radau weights are computed by Theorem
\ref{th:laguerre_radau}. Numerical results are shown in Figure
\ref{fourthfig}. As we can see, both formulas are of approximately
equal accuracy.

\section{Conclusion}\label{s:remarks}

We have investigated the fast computation of the
interpolation polynomials based on the zeros or extrema of classical
families of orthogonal polynomials. We have shown that the
barycentric weights and the corresponding quadrature weights are
intimately related to each other and that such relationships are a direct
consequence of the existence of lowering operators for orthogonal polynomials. Note
that the nodes and weights of the classical Gaussian quadrature formulas can
be efficiently computed using the Glaser-Liu-Rokhlin algorithm for Laguerre and Hermite polynomials, and by the more efficient Hale-Townsend algorithm \cite{hale2012fastgauss} for the Jacobi polynomials. The interpolation polynomials based on the zeros of these polynomials can thus be computed efficiently by using their barycentric representations.

The formulas for the barycentric weights for the Jacobi, Laguerre and Hermite polynomials were already described and implemented as part of the Chebfun package \cite{trefethen2011chebfun4,hale2012chebfunquadrature}. We have extended the idea to the implementation of the barycentric interpolation in the extrema of these classical
polynomials with some additional boundary points, e.g. Gauss-Radau
and Gauss-Lobatto points. The link between the barycentric weights
and the corresponding quadrature weights is established which allows
the computation of the interpolants in Gauss-Radau and Gauss-Lobatto
points in $\mathcal{O}(n)$ operations as well.

\section*{Acknowledgement}
The authors would like to thank Alfredo Dea\~{n}o and Lun Zhang for
helpful discussions about the theory of lowering operators for
orthogonal polynomials, and Jean-Paul Berrut and Nick Trefethen for their valuable comments on the history and recent developments of barycentric weights and their implementation in Chebfun.



\end{document}